\documentclass[11pt]{article}
\usepackage{amsmath}
\usepackage{latexsym}
\usepackage{graphicx}
\usepackage{amssymb}
\begin{document}
\makeatletter
\def\@begintheorem#1#2{\trivlist \item[\hskip \labelsep{\bf #2\ #1.}] \it}
\def\@opargbegintheorem#1#2#3{\trivlist \item[\hskip \labelsep{\bf #2\ #1}\ {\rm (#3).}]\it}
\makeatother
\newtheorem{thm}{Theorem}[section]
\newtheorem{alg}[thm]{Algorithm}
\newtheorem{conj}[thm]{Conjecture}
\newtheorem{lemma}[thm]{Lemma}
\newtheorem{defn}[thm]{Definition}
\newtheorem{cor}[thm]{Corollary}
\newtheorem{exam}[thm]{Example}
\newtheorem{prop}[thm]{Proposition}
\newenvironment{proof}{{\sc Proof}.}{\rule{3mm}{3mm}}

\title{On the chromatic index\\ of generalized truncations}
\author{Brian Alspach and Aditya Joshi\\School of Mathematical and Physical Sciences\\
University of Newcastle\\Callaghan, NSW 2308,
Australia\\brian.alspach@newcastle.edu.au\\aditya.joshi@uon.edu.au}
\maketitle

\begin{abstract} We examine the chromatic index of generalized truncations of
graphs and multigraphs.
\end{abstract}


\medskip


\section{Introduction}

A broad definition of generalized truncations of graphs was introduced in \cite{A3}.  We
give this definition now for completeness, but first a brief word about terminology is in order.
The term {\it multigraph} is used if multiple edges are allowed.  Thus, a graph does not
have multiple edges.  We use $V(X)$
to denote the set of vertices of a multigraph $X$ and $E(X)$ to denote the set of
edges.  The {\it order} of $X$ is $|V(X)|$ and the {\it size} of $X$ is $|E(X)|$.  Finally,
the {\it valency} of a vertex $u$, denoted $\mathrm{val}(u)$, is the number of edges
incident with $u$.  The term $k$-{\it valent} multigraph is used for regular multigraphs
of valency $k$.  Throughout the paper, $\Delta(X)$ denotes the maximum valency of the
multigraph $X$ and $\Delta$ often is used when the graph $X$ involved is apparent.

Given a multigraph $X$, a {\it generalized truncation} of $X$ is obtained as follows via a
two-step operation. The first step is the {\it excision step}.  Let $M$ denote an auxiliary 
matching (no two edges have a vertex in common) of size $|E(X)|$.  Let
$F:E(X)\rightarrow M$ be a bijective function and for $[u,v]\in E(X)$, label the ends of the
edge $F([u,v])$ with $u$ and $v$.  Let $M_F$ denote the vertex-labelled matching thus obtained.
So $M_F$ represents the edges of $X$ completely disassembled.

The second step is the {\it assemblage step}.  For each $v\in V(X)$, the set of vertices of
$M_F$ labelled with $v$ is called the {\it cluster at} $v$ and is denoted $\mathrm{cl}(v)$.  
Insert an arbitrary graph on $\mathrm{cl}(v)$.  The inserted graph on $\mathrm{cl}(v)$ is
called the {\it constituent graph at $v$} and is denoted $\mathrm{con}(v)$. The resulting
graph $$M_F\cup_{v\in V(X)} \mathrm{con}(v)$$
is a {\it generalized truncation} of $X$.  We usually think of the labels on the vertices of
$F(M)$ as being removed following the assemblage stage, but there are many times when the labels
are useful in the exposition.  We use $\mathrm{TR}(X)$ to denote a
generalized truncation of the multigraph $X$.

Truncations arise via action involving the edges incident with a vertex.  Consequently, isolated
vertices are useless and we make the important convention that the multigraphs from
which we are forming generalized truncations do not have isolated vertices.  This will not be
mentioned in the subsequent material, but is required for the validity of a few statements.

Recall that a {\it proper edge coloring} of a multigraph $X$ is a coloring of the edges so that
adjacent edges do not have the same color.  The {\it chromatic index} of $X$, denoted
$\chi'(X)$, is the fewest number of colors for which a proper edge coloring exists.  We now
state Vizing's well-known theorem \cite{V1} which plays a fundamental role in studying
chromatic index.  

\begin{thm}\label{viz} If $X$ is a graph, then its chromatic index is either $\Delta(X)$ or
$\Delta(X)+1$.
\end{thm}

The preceding theorem leads to a classification of graphs as follows.  A graphs is {\it class} I
if its chromatic index is equal to its maximum valency $\Delta$ and is {\it class} II otherwise.
There is a generalization for multigraphs and to ease subsequent exposition we shall say
a multigraph is class I if its chromatic index equals its maximum valency.

The following result was proved in \cite{A3}.  A generalized truncation $\mathrm{TR}(X)$
is called {\it complete} when every constituent is a complete subgraph.

\begin{thm}\label{color} If $X$ is a class {\em I} graph, then its complete truncation also is
class {\em I}.  If $X$ is a class {\em II} graph and its maximum valency is even, then its
complete truncation is class {\em I}.
\end{thm}

The purpose of this paper is to extend the preceding result and explore edge colorings of
generalized truncations in more detail.

\section{Some Useful Results}

In the following material we shall be considering the relationship between multigraphs and
their generalized truncations that are class I.
Doing so requires the use of some well-known edge coloring results which we now give for
completeness. 

Let $\sigma$ be a permutation  of the vertex set of the complete graph $K_n$ of order $n$.
If $Y$ is a subgraph of $K_n$, then $\sigma(Y)$ denotes the subgraph of $K_n$ whose edge
set is $\{[\sigma(u),\sigma(v)]:[u,v]\in E(Y)\}$.  Let $n$ be even and $\rho$ be the permutation
defined by $\rho=(u_0)(u_1\;\;u_2\;\;u_3\;\cdots\;u_{n-1})$.  So $\rho$ fixes one vertex and
cyclically rotates the remaining vertices.  Let $Y$ be the perfect matching of $K_n$ consisting of
the edges \[[u_0,u_1],[u_2,u_{n-1},[u_3,u_{n-2}],\ldots,[u_{n/2},u_{(n+2)/2}].\] We then
obtain a proper edge coloring of $K_n$ by letting the color classes be $Y,\rho(Y),\rho^2(Y),
\ldots,\rho^{n-2}(Y)$.  We call this proper edge coloring the {\it canonical edge coloring} of
$K_n$.  Figure 1 shows $Y$ for $K_8$.

\begin{picture}(100,140)(-130,-30)
\multiput(20,0)(20,0){2}{\circle*{5}}
\multiput(0,30)(60,0){2}{\circle*{5}}
\multiput(0,60)(60,0){2}{\circle*{5}}
\put(30,80){\circle*{5}}
\put(30,40){\circle*{5}}
\put(30,40){\line(0,1){40}}
\put(0,30){\line(1,0){60}}
\put(0,60){\line(1,0){60}}
\put(20,0){\line(1,0){20}}
\put(10,-20){\sc Figure 1}
\put(35,40){\small $u_0$}
\put(35,80){\small $u_1$}
\put(65,60){\small $u_2$}
\put(65,30){\small $u_3$}
\put(45,3){\small $u_4$}
\put(5,3){\small $u_5$}
\put(-15,30){\small $u_6$}
\put(-15,60){\small $u_7$}
\end{picture}

We obtain a proper edge coloring of $K_n$ with $n$ colors, $n$ odd, by starting with the
canonical edge coloring of $K_{n+1}$ and then removing the central vertex $u_0$ and the
edges incident with it.  Note that each vertex is missing an edge of one color and the set
of missing colors has cardinality $n$.  This fact is true for any proper edge coloring of $K_n$,
$n$ odd.

Because we use the preceding canonical edge coloring for both even and odd orders, we
describe it by referring to the edge of length 1 as the {\it anchor}.   When $n$ is even, the
anchor is the edge $[u_{n/2},u_{1+n/2}]$.  When $n$ is odd, the anchor is the edge
$[u_{(n+1)/2},u_{(n+3)/2}]$.  To obtain proper edge colorings of complete graphs,
we cyclically rotate the canonical edge coloring which, of course, cyclically rotates the
anchor.  So we may determine a color class by specifying the anchor.

\begin{lemma}\label{unique} Every proper edge coloring of $K_n$ with $n$ colors, $n$ odd,
has the property that there is one color missing on the edges incident with a given vertex,
and the set of missing colors at the vertices has cardinality $n$.
\end{lemma}
\begin{proof} There are at most $(n-1)/2$ edges of a fixed color in a proper edge coloring
of $K_n$ because $n$ is odd.  Thus, there is no proper edge coloring using just $n-1$
colors because $(n-1)^2/2<\binom{n}{2}$.  There is a proper edge coloring using $n$
colors by Vizing's Theorem.  Hence, each color class contains precisely $(n-1)/2$ edges
from which the conclusion follows.   \end{proof}

\medskip

We use a result about list chromatic index so that we discuss it briefly here and present
the result.  Given a graph $X$ and for each edge $[u,v]$ a list $L_{[u,v]}$ of colors we may
use for the edge $[u,v]$, a {\it proper list edge coloring} is a proper edge coloring of $X$
so that the color on each edge $[u,v]$ belongs to $L_{[u,v]}$.  The {\it list chromatic index}
of $X$ is the smallest $N$ such that if every list $L_{[u,v]}$ has cardinality $N$, then
$X$ admits a proper list edge coloring.  We denote this value by $\chi'_L(X)$.
The following important result was proved by H\"aggkvist and Jansen in \cite{H1}.

\begin{thm}\label{HJ} The complete graph $K_n$ satisfies $\chi'_L(K_n)\leq n$.
\end{thm}

\section{Main Result}

We are interested in determining which generalized truncations of a given multigraph
$X$ are class I.  Theorem \ref{vector1} below provides a general answer and in subsequent
sections we use the theorem to describe class I generalized truncations of specific types.
Some definitions and notation are required before stating the theorem.

Given a generalized truncation $\mathrm{TR}(X)$, for convenience
we call the subgraph induced by the edges of $\mathrm{con}(v)$ together with the
edges of $M_F$ incident with the vertices of $\mathrm{con}(v)$ the {\it sun centered at}
$\mathrm{con}(v)$.  Given that $\mathrm{con}(v)$ is regular of valency $d-1$, when
is the sun centered at $\mathrm{con}(v)$ class I?  If it is class I, then there is a
proper edge coloring using $d$ colors.  We use a vector of length $d$ to describe the number
of edges of $M_F$ of the various colors in the sun centered at $\mathrm{con}(v)$.   

Given a vertex $v\in X$ of valency $r$, we say that a vector $(x_1,x_2,\ldots,x_d)$
is {\it admissible} if $\sum_i x_i=r$, the edges of $M_F$ are colored according to the
vector, and there is a $(d-1)$-regular graph on $\mathrm{con}(v)$ such that the sun
centered at $\mathrm{con}(v)$ is class I.  We say the vector is {\it totally inadmissible} if
there is no $(d-1)$-regular graph on $\mathrm{con}(v)$ which makes the sun class I. 

\begin{thm}\label{vector1} If $(x_1,x_2,\ldots,x_d)$ satisfies $x_1+x_2+\cdots+x_d=r
\geq d$, then $(x_1,x_2,\ldots,x_d)$ is admissible if and only if all of
$x_1,x_2,\dots,x_d\mbox{ and }r$ have the same parity and when the parity is odd,
$d$ also is odd.  Otherwise, the vector is totally inadmissible.  
\end{thm}
\begin{proof} Suppose that $\mathrm{val}(v)=r$ in a multigraph $X$ and that there
is a class I sun centered at $\mathrm{con}(v)$ which is regular of valency $d$.  Let
$(x_1,x_2,\ldots,x_d)$ be the vector for the colors of $M_F$.  For an arbitrary color
$c(i)$, if the number of edges of color $c(i)$ in $\mathrm{con}(v)$ is $k$, then the
number of edges of color $c(i)$ in $M_F$ must be $r-2k$.  It then follows that
$x_1,x_2,\ldots,x_d\mbox{ and }r$ all have the same parity.  Moreover, when $r$ is
odd, then the graph $\mathrm{con}(v)$ is regular of valency $d-1$ and order $r$
which implies that $d-1$ is even.  This completes the necessity.

Now suppose that all coordinates of $(x_1,x_2,\ldots,x_d)$ are odd, $r$ is odd and
$x_1+x_2+\cdots+x_d=r$.  In this case we also have that $d$ is odd.  Label the $r$
vertices of the constituent $u_1,u_2,u_3,\ldots,u_r$ so that the first $x_1$ vertices
are incident with the $x_1$ edges of $M_F$ of color $c(1)$, the next $x_2$ vertices
are incident with the $x_2$ edges of $M_F$ of color $c(2)$ and continue in this way.
Also carry out subscript arithmetic modulo $r$ on the residues $1,2,3,\ldots,r$.

Consider the $x_i$ successive vertices incident with the $x_i$ edges of $M_F$ of color
$c(i)$.  Because $x_i$ is odd, there is a central edge of $M_F$ of color $c(i)$ and let
it be incident with $u_a$.  Letting $\alpha=(x_i-1)/2$, the vertices incident with edges
of $M_F$ of color $c(i)$ are $u_{a-\alpha},u_{a-\alpha+1},\ldots,u_a,u_{a+1},\ldots,
u_{a+\alpha}$.  

Consider the canonical color class whose anchor is $[u_{a+(r-1)/2},
u_{a+(r+1)/2}]$.  Color the edges of this class with color $c(i)$ starting with the
anchor and stopping with the edge $[u_{a-\alpha-1},u_{a+\alpha+1}]$.  This yields
edges colored with $c(i)$ so that each vertex of the constituent is incident with one
edge of color $c(i)$.  After doing this for each of the $d$ colors, the resulting sun
centered at the constituent is class I and the constituent is regular of valency $d-1$.

\smallskip

This leaves us with the case that all the coordinates of $(x_1,x_2,\ldots,x_d)$ are
even so that their sum $r$ also is even.  There is no restriction on the parity of $d$
in this case. Also note that $x_i=0$ is possible for various values of $i$.  Without loss
of generality, let $x_1,x_2,\ldots,x_a$ be the non-zero entries of the vector.

We label 
the vertices of the constituent a little differently using the canonical
edge coloring of Figure 1 as a template.  The central vertex is labelled $u_0$ and the
others are labelled $u_1,u_2,\ldots,u_{r-1}$.

Let the $x_1$ edges of $M_F$ of color $c(1)$ be incident with $u_0,u_1,\ldots,u_{x_1-1}$. 
Let the remaining edges of $M_F$ be incident with vertices of the constituent as in
the preceding case, that is, edges of the same color are incident with successively
labelled vertices.  

The edges of color $c(1)$ have a different form than the other colored edges making
the completion of the coloring a little different for them.  Vertices $u_1,u_2,\ldots,u_{x_1-1}$
are incident with edges of $M_F$ of color $c(1)$.  This corresponds to the color
class with anchor $[u_{(x_1+r-4)/2},u_{(x_1+r-2)/2}]$.  So color the edges of that
class with color $c(1)$ starting with the anchor until finishing with $[u_{r-1},u_{x_1}]$. 
Every vertex of the constituent now is incident with an edge of color $c(1)$.

For all the other colors $c(j)$, $1<j\leq a$, there are an even number of successive vertices,
say $u_j,u_{j+1},\ldots,u_{j+t}$, incident with edges of $M_F$ of color $c(j)$.  Then
the edge $[u_{j+(t-1)/2},u_{j+(t+1)/2}]$ is the anchor of a color class.  So complete
the coloring of the edges of this color class starting with $[u_{j-1},u_{j+t+1}]$ moving
away from the anchor.  For the $x_i=0$, simply include an entire color class from the
canonical coloring scheme using an unused anchor.  

Doing the preceding yields a class I coloring of the sun centered
at the constituent so that the constituent graph is regular of valency $d-1$ and completes
the proof.   \end{proof} 

\medskip

We now obtain three corollaries from Theorem \ref{vector1}, but first require a couple of
definitions.  An edge coloring of a multigraph $X$ is said to be {\it parity-balanced} if
for each vertex $v$ of $X$, the parity of the number of edges of each color incident with
$v$ is the same as the parity of $\mathrm{val}(v)$.  A {\it regular truncation} is a
generalized truncation for which every vertex has the same valency.   A generalized truncation
is said to be {\it semiregular} if each sun centered at a constituent $\mathrm{con}(v)$ is
class I and the subgraph $\mathrm{con}(v)$ is regular.  The distinction between a
semiregular truncation and a regular truncation is that valencies of regularity for a
semiregular truncation may differ over the constituents.  Note that the source
multigraph need not be regular in order to have a regular truncation.

\begin{cor}\label{even} Let every entry of the feasible vector $(x_1,x_2,\ldots,x_d)$ be even and
$x_1+x_2+\cdots+x_d=r$.  If $d'$ is the number of non-zero entries of the vector,
then there are class {\em I} suns centered at the appropriate constituent such that the constituents
are regular of every valency from $d'$ through $r-1$.
\end{cor}
\begin{proof} This result follows from the proof of Theorem \ref{vector1} rather than the
statement.  Because the number of colors on edges of $M_F$ is $d'$, the scheme used in
the proof of Theorem \ref{vector1} introduces no new colors so that the valency of
regularity is $d'$.  Because $r$ is even, we may add canonical coloring classes one at a
time until reaching $K_r$.  This completes the proof.   \end{proof} 

\begin{cor}\label{sreg} A multigraph $X$ has a class {\em I} semiregular truncation if and only if
$X$ has a parity-balanced edge coloring.
\end{cor}
\begin{proof}  If $X$ has a parity balanced edge coloring, it follows immediately from
Theorem \ref{vector1} that it has a semiregular truncation.  On the other hand, if $X$
has a semiregular truncation, then performing the standard contraction and retention
of the colors on the edges of $M_F$ produces a parity-balanced edge coloring of $X$.   \end{proof}

\begin{cor}\label{reg} A multigraph $X$ has a class {\em I} regular truncation of valency $d$ if and only if it
has a parity-balanced coloring and one of the following conditions holds:\\

{\rm (i)} When $d$ is odd, every vertex of odd valency has precisely $d$ colors on its
incident edges, and every vertex of even valency has valency at least $d+1$ and at
most $d$ colors on its incident edges ; or\\

{\rm (ii)} When $d$ is even, every vertex of $X$ has even valency at least $d$.
\end{cor}
\begin{proof} Let $X$ be a mutigraph and suppose it has a generalized truncation $Y$
which is regular of valency $d$.  Consider the case that $d$ is even.  This means that a
constituent $\mathrm{con}(v)$ is regular of valency $d-1$ and the latter is odd.  Thus,
the constituent has even order.  This implies that every vertex of $X$ has even valency.
Clearly every valency is at least $d$.  

When $d$ is odd, each constituent must itself be regular of valency $d-1$.  Because
$d-1$ is even, there is no restriction on the order $N$ of the constituent other than it must
be at least $d$.  If $N$ is odd, then for each of the $d$ distinct colors, there must be at
least one edge of that color belonging to the edges of $M_F$ incident with vertices of the
constituent.  So the corresponding vertex of $X$ is incident with edges of precisely $d$
different colors.  When $N$ is even there is no restriction on the number of colors on
edges of $M_F$ incident with vertices of the constituent other than it is at most $d$.

\smallskip

For the other direction, when $d$ is even, every constituent has even order and by
Corollary \ref{even} it is easy to obtain regular valency of $d-1$ on the constituent.
When $d$ is odd, the result follows from Theorem \ref{vector1}.   \end{proof}

\section{Complete Truncations}

Let $X$ be a multigraph with maximum valency $\Delta$.  If $\Delta$ is odd and $X$ has an
edge coloring with $\Delta$ colors such that the edges incident with every vertex of
valency $\Delta$ have distinct colors, then we say that $X$ is {\it edge-feasible}. 
The next result characterizes graphs whose complete truncations are class I.  Of course,
a complete truncation is a semi-regular truncation.

\begin{thm}\label{main} The complete truncation of a multigraph $X$ is class {\em I}
if and only if either the maximum valency $\Delta$ of $X$ is even, or $\Delta$ is odd
and $X$ is edge-feasible.
\end{thm}
\begin{proof} This theorem is an improvement on Theorem \ref{color} as the latter does
not include multigraphs as part of the hypotheses.  Because a complete truncation is a
semiregular truncation, Corollary \ref{sreg} implies that the coloring of the edges in $M_F$
must correspond to a parity-balanced coloring of $X$.  When $\Delta$ is even, coloring all the
edges of $M_F$ with a single color corresponds to a parity balanced coloring of $X$.  Any
constituent of order $\Delta$ can be properly edge-colored with $\Delta-1$ colors because
$\Delta$ is even.  Any constituent of order less than $\Delta$ can be properly edge-colored
with at most $\Delta-1$ colors.  Hence, the complete truncation of $X$ is class I.

For the remainder of the proof we assume that $\Delta$ is odd.  First suppose that the
complete truncation $\mathrm{TR}(X)$ is class I.
If $u\in V(X)$ has valency $\Delta$, then $\mathrm{con}(u)=K_{\Delta}$ the complete
graph of order $\Delta$. We conclude that the edges of $M_F$ incident with $\mathrm{con}(u)$
all have different colors by Lemma \ref{unique}.  So if we contract each  constituent to a
single vertex, delete all the loops formed and keep the colors of the edges of $M_F$, we
obtain an edge coloring of $X$.  Clearly, all the edges incident to any vertex of valency
$\Delta$ in $X$ have distinct colors.  Thus, $X$ is edge-feasible.

Now let $X$ be edge-feasible and choose an edge coloring of $X$ which is edge-feasible.
Let $\mathrm{TR}(X)$ be the complete truncation of $X$.  Color the edges
of $M_F$ with the same color they had in the edge coloring of $X$.  These are the edges
between the constituents all of which are complete subgraphs.  At this point we have used
$\Delta$ colors.  It suffices to show that we can color the edges of the constituents without
introducing any new colors so that the resulting edge coloring of $\mathrm{TR}(X)$ is proper. 

Any constituent $\mathrm{con}(u)$ of order $\Delta$ corresponds to a vertex $u$ of
$X$ of valency $\Delta$.  This implies that the edges of $M_F$ incident with the vertices
of $\mathrm{con}(u)$ all have distinct colors because the coloring of $X$ was edge-feasible.
From Lemma \ref{unique} it is clear that we may color the edges of $\mathrm{con}(u)$
with $\Delta$ colors so that the missing color at each vertex is the color of the edge of
$M_F$ at the vertex.    This coloring of the edges does not violate the definition of a
proper edge coloring.

Now consider any constituent $\mathrm{con}(u)$ of order $r\leq\Delta-2$.  Because
each edge of $\mathrm{con}(u)$ has one edge of $M_F$ at each of its end vertices,
the number of possible colors for the edge that do not violate the proper edge
coloring condition is $\Delta-2$.  So each edge has a list of $\Delta-2$ possible
colors, and Theorem \ref{HJ} implies that we may color the edges of $\mathrm{con}(u)$
without violating the proper edge coloring condition because $r\leq\Delta-2$.

This leaves us with the case that $\mathrm{con}(u)$ has order $\Delta-1$ and this is
the most complicated case.  The first observation we make is that the coloring pattern 
of the edges of $M_F$ incident with the vertices of $\mathrm{con}(u)$ can vary all
the way from having $\Delta-1$ distinct colors to every color being the same.  So we
introduce a sequence describing the color pattern.  Let $(s_1,s_2,\dots,s_t)$ satisfy
$s_1\leq s_2\leq\cdots\leq s_t$ and $s_1+s_2+\cdots+s_t=\Delta-1$.  The sequence
means there are $t$ distinct colors on the edges of $M_F$ incident with vertices of
$\mathrm{con}(u)$ and $s_i$ such edges have the same color $c(i)$ for $i=1,2,\ldots,t$.

We need to show that no matter what the color sequence is we may color the edges
of $\mathrm{con}(u)$ so that the sun centered at $\mathrm{con}(u)$ is properly
edge colored with $\Delta$ colors.  As a first step, label the vertices of $\mathrm{con}(u)$
with $v_0,v_1,\ldots,v_a$, where $a=\Delta-2$.  Let the vertices $v_1,v_2,\ldots,v_{s_1}$
be incident with the $s_1$ edges of $M_F$ of color $c(1)$.  Let the next $s_2$ vertices
be incident with the $s_2$ edges of $M_F$ of color $c(2)$.  Continue labelling the
vertices in the obvious manner and let the edge of $M_F$ incident with $v_0$ have
color $c(t)$. 

Carry out an initial coloring of the edges of $\mathrm{con}(u)$ using the canonical
edge coloring described in Section 2 with $v_0$ acting as the fixed vertex and $\rho=
(v_0)(v_1\;\;v_2\;\;\cdots\;\;v_a)$.  The strategy now is to choose the colors for
the color classes in such a way that we may re-color some edges to obtain a proper
edge coloring for the sun centered at $\mathrm{con}(u)$.   

The first observation we make is that only $\Delta-2$ colors are required for a
canonical edge coloring of $\mathrm{con}(u)$.  So we do not use the colors $c(t-1)$
or $c(t)$ for the canonical edge coloring of $\mathrm{con}(u)$.  Hence, if $t=1$
or $t=2$, we already have an edge coloring of $\mathrm{con}(u)$ that does not
violate the conditions for a proper edge coloring.  Thus, we assume that $t\geq 3$.

The anchor for a color class plays an active role as follows.  Suppose we require two
edges of a color class so that the two edges use four successive vertices under the
cyclic labelling of $\{v_1,v_2,\ldots,v_a\}$.  If the anchor is $[v_i,v_{i+1}]$, then
adding the edge $[v_{i-1},v_{i+2}]$ from the same color class easily does the
job.  It is now easy to see how we may obtain $k$ edges from the same color class
so that they cover $2k$ successive vertices.  With this in mind, we now describe an
iterative process for determining the colors of certain color classes.  

If $s_1$ is odd, then we color the color class for which $[v_{(s_1+1)/2},v_{(s_1+3)/2}]$
is the anchor with $c(1)$.  If $s_1$ is even, then we color the color class for which
$[v_{s_1/2},v_{(s_1+2)/2}]$ is the anchor with $c(1)$.  There are two points to
observe about the preceding choices.  When $s_1$ is odd, the edge from $v_1$ to
$v_{s_1+1}$ has color $c(1)$ and note that the edge of $M_F$ incident with
$v_{s_1+1}$ has color $c(2)$.  When $s_1$ is even, the edge from $v_1$
to $v_{s_1}$ has color $c(1)$ and vertex $v_{s_1}$ is the last vertex for which the
edge of $M_F$ incident with it has color $c(1)$.

We continue in the manner suggested by the preceding paragraph, but discuss it
further to make it clearer.  When $s_i$ is odd and the first vertex incident with an
edge of $M_F$ of color $c(i)$ is $v_d$, then the anchor is the edge $[v_{d+(s_i-1)/2},
v_{d+(s_i+1)/2}]$ and we color this color class with $c(i)$.  Note that the edge
$[v_d,v_{d+s_i}]$ is in this color class, but the edge of $M_F$ incident with
$v_{d+s_i}$ has color $c(i+1)$.

When $s_i$ is even and the first vertex incident with an edge of $M_F$ of color
$c(i)$ is $v_d$, then the anchor is $[v_{d+(s_i-2)/2},v_{d+s_i/2}]$ and we color
this color class with $c(i)$.  In this case, the edge from $v_d$ to $v_{d+s_i-1}$
is colored $c(i)$ and $v_{d+s_i-1}$ is the last vertex incident with an edge of
$M_F$ of color $c(i)$. 

The preceding procedure is carried out for the colors $c(1)$ through $c(t-2)$ at
which point it stops because colors $c(t-1)$ and $c(t)$ are not used for the canonical
edge coloring of $\mathrm{con}(u)$.  The edges of $\mathrm{con}(u)$ that need
to be re-colored are those that are adjacent with edges of $M_F$ having the same
color.   We have seen that when $s_i$ is odd, there is an edge in $\mathrm{con}(u)$
of color $c(i)$ with one end vertex incident with an edge of $M_F$ of color $c(i)$ and
the other end vertex incident with an edge of $M_F$ of color $c(i+1)$.  So at the
end vertex incident with an edge of $M_F$ of color $c(i+1)$, the procedure gives
another edge of color $c(i+1)$ which also needs to be re-colored.  

Thus, the edges that need to be recolored are isolated and possibly paths.
Luckily we have two colors to use for the re-coloring and we arbitrarily re-color
isolated edges with either $c(t-1)$ or $c(t)$, and alternately re-color the edges
of the paths making certain that if one of the paths terminates with a vertex
whose incident edge from $M_F$ has color $c(t-1)$, we color the edge of the
path terminating there with $c(t)$.  This removes all potential color conflicts
for $\mathrm{con}(u)$ and completes the proof of the theorem.      \end{proof}

\begin{cor}\label{Delta} Let $X$ be a class I graph and $\mathrm{TR}(X)$ a generalized
truncation.  If $\Delta(X)=\Delta(\mathrm{TR}(X))$, then $\mathrm{TR}(X)$ is
class I.
\end{cor}
\begin{proof} Let $\mathrm{TR}(X)$ be a generalized truncation of $X$ satisfying
$\Delta(\mathrm{TR}(X))=\Delta(X)$.  Then a proper edge coloring of $\mathrm{TR}(X)$
requires at least $\Delta$ colors.  We know the complete truncation $Y$ of $X$
is class I by Theorem \ref{main}, that is, it has a proper edge coloring using
$\Delta$ colors.  Remove any edges of $Y$ not belonging to $\mathrm{TR}(X)$ and
we are left with a proper edge coloring of $\mathrm{TR}(X)$ using $\Delta$ colors.  So
$\mathrm{TR}(X)$ is class I.   \end{proof}

\medskip

It is well known that both the Petersen graph and its complete truncation are class II graphs.
The extension to a regular multigraphs is an immediate corollary of Theorem \ref{main}.

\begin{cor}\label{2-2} A regular multigraph $X$ of odd valency is class {\em I} if and only
if its complete truncation is class {\em I}.
\end{cor}

\section{Cyclic Truncations}

Complete truncations have been used by various authors, and there is
another generalized truncation that has been employed frequently. Namely, the generalized
truncation obtained by letting each constituent graph be a cycle.  We shall call these
{\it cyclic truncations}.  Probably the best known example of this is the cube-connected
cycles graph first introduced in \cite{P1}.  Of course, the ancient Greeks studied truncations
of Platonic and Archimedian solids, and these resulted in cyclic truncations.

When moving from a multigraph $X$ to a generalized truncation $Y$, we frequently
wish $Y$ to inherit properties of $X$.  This can be problem for cyclic truncations.  For
example, poorly chosen cyclic truncations can play havoc with automorphisms.

We are interested in determining conditions for which a cyclic truncation is class I.
Of course, a cyclic truncation is a regular truncation so that we may use Theorem \ref{vector1}.
We use the same vector describing the numbers of colors used on the edges
of $M_F$ belonging to a sun centered at a constituent.  In this case the vector has length
3 as we are looking at cyclic truncations all of which are 3-valent.  There is a new term
not used before, namely, a vector $(x_1,x_2,x_3)$ is {\it universal} if the sun centered
at the corresponding constituent is class I for every cycle on the vertices of the constituent.
 
\begin{cor}\label{vector} If $(x_1,x_2,x_3)$ satisfies $x_1+x_2+x_3=d\geq 3$, then
$(x_1,x_2,x_3)$ is admissible if and only $x_1,x_2,x_3\mbox{ and }d$ all have the same
parity.  Otherwise, the vector is totally inadmissible.  Moreover, if just one of $x_1,x_2
\mbox{ and }x_3$ is non-zero, then $(x_1,x_2,x_3)$ is universal when $d$ is even.
\end{cor} 
\begin{proof} The portion of the statement prior to the sentence about universality is
simply Theorem \ref{vector1} restricted to $d=3$.  When the vector has the form
$(x_1,0,0)$ and $x_1=d$ is even, this corresponds to all the edges of $M_F$ incident
with the vertices of the constituent having the same color.  Clearly, no matter which
even length cycle we form on the vertices of the constituent, the resulting sun is 
class I.     \end{proof}

\medskip

Corollary \ref{vector} gives us an easy way to construct a class I cyclic truncation of a
multigraph $X$ based on parity conditions for the colors on the edges of $M_F$.  Thus,
we need to concentrate on coloring the edges in the source multigraph.  

\begin{cor} If every vertex of the multigraph $X$ has even valency, then every cyclic
truncation of $X$ is class I.
\end{cor}
\begin{proof} This follows from Corollary \ref{vector} by coloring all edges of $M_F$
with a single color.    \end{proof}

\begin{cor}\label{reg} If $X$ is a regular multigraph of odd valency $d>1$ and is class 
{\em I}, then there are class {\em I} cyclic truncations of $X$. 
\end{cor}
\begin{proof} Let $X$ be a class I regular multigraph of odd valency $d\geq 3$.  Choose a
proper edge coloring of $X$ using $d$ colors.  In forming a cyclic truncation $Y$ of
$X$, color the edges of $M_F$ according to the colors $c(1),c(2),\ldots,c(d)$ the
edges have in the proper edge coloring of $X$.  Now change the color of any edge
of colors $c(4),c(5),\ldots,c(d)$ to $c(3)$.  The vector for each constituent becomes
$(1,1,d-2)$ all of which are odd.  We now may find a cycle for each constituent such
that the cyclic truncation is class I by Corollary \ref{vector}.   \end{proof}

\medskip

We now give a sufficient condition for a multigraph to have a class I cyclic truncation.
A definition is required.  Let $X$ be an arbitrary multigraph and $V_0,V_1,V_2\mbox{ and }
V_3$ be a partition of $V(X)$, where $V_i$ contains the vertices whose valencies
are congruent to $i$ modulo 4.  Given a submultigraph $Y$ of $X$, let $X\setminus Y$
be the submultigraph obtained by removing the edges of $Y$ from $X$.  Moreover, let
$V_0',V_1',V_2'\mbox{ and }V_3'$ denote the vertices of $V(X)$ whose valencies in
$X\setminus Y$ are congruent to $i$ modulo 4.  A submultigraph $Y$ of $X$ is called an {\it enabling
submultigraph} if 
satisfies: \begin{itemize}
\item if $v\in V_0$, then $v\in V_0'$;
\item if $v\in V_1$, then $v\in V_2'$;
\item if $v\in V_2$, then $v\in V_0'$; and
\item if $v\in V_3$, then $v\in V_2'$.
\end{itemize} 

The preceding conditions are not as complicated as they may look at first.  The first
thing to notice is that the components of $X\setminus Y$ are eulerian.  Second, for
many graphs the conditions are simple.  For example, if $X$ is 3-regular, then a
perfect matching is an enabling subgraph.

\begin{cor} Let $X$ be a multigraph with an enabling submultigraph $Y$.  If every
component of $X\setminus Y$ has even size, then there is a class I cyclic truncation of $X$.
\end{cor}
\begin{proof} Because each component of $X\setminus Y$ has even size, we may alternately
color the edges of an Euler tour of the component with two colors so that we finish with same
number of colors on the edges incident with every vertex.  Doing this for each component
yields a 2-coloring so that each vertex of $X\setminus Y$ has the same number of colors
incident with it.

If we now color all the edges of $Y$ with a third color, it is easy to verify that the
conditions of Corollary \ref{vector} are met for $X$.  We check one possibility for
illustrative purposes.  If $v\in V_3$, it belongs to $V_2'$ in $X\setminus Y$.  Thus, the
number of edges of $Y$ incident with $v$ is congruent to 1 modulo 4, that is, it is
incident to an odd number of edges of the third color in $X$.  Because it is in $V_2'$,
it is incident with an odd number of edges of each of the first two colors.  Thus, the
vector for $v$ has three odd components.     \end{proof}

\begin{prop}\label{ce}  If a trivalent graph $X$ has a cut edge, then $X$ is class {\em II}.
\end{prop}

The proof of the preceding proposition is immediate and leads to examples such as
the following.
Let $X$ be the graph of order 10 obtained by taking two vertex-disjoint copies of
$K_5$ and joining the two copies with a single edge from a vertex of one copy to
a vertex of the other copy.  Lemma \ref{unique} implies that $X$ is class I.  If we
take any cyclic truncation $\mathrm{TR}(X)$ of $X$, then $\mathrm{TR}(X)$ is
class II by Proposition \ref{ce} because the edge of $M_F$ connecting the two copies still
is a cut edge in $\mathrm{TR}(X)$.

\section{Other Truncations}

We now examine other generalized truncations.  An {\it arboreal truncation} is a
generalized truncation for which every constituent graph is a forest.  The following
result is useful for arboreal truncations.

\begin{thm}\label{strong} Let $\mathrm{TR}(X)$ be a generalized truncation of a
multigraph $X$.  If the maximum valency of $\mathrm{TR}(X)$ is $\Delta$ and every
constituent of $\mathrm{TR}(X)$ having a vertex of valency $\Delta$ is class {\em I},
then $\mathrm{TR}(X)$ is class {\em I}.
\end{thm}
\begin{proof} Let $\mathrm{con}(v)$ have a vertex of valency $\Delta$ in $\mathrm{TR}(X)$.
Then the maximum valency in the subgraph $\mathrm{con}(v)$ is $\Delta-1$.  Its edges
may be properly edge-colored with $\Delta-1$ colors because it is class I.  Any constituent
not having a vertex of valency $\Delta$ may have its edges properly colored with at most 
$\Delta-1$ colors because its maximum valency is $\Delta-2$.  Then the edges of $M_F$
may be colored with a new color yielding a proper edge coloring of $\mathrm{TR}(X)$
with $\Delta$ colors.  The conclusion now follows.     \end{proof}

\begin{cor}\label{} Every arboreal truncation of a multigraph $X$ is class I.
\end{cor}
\begin{proof}  This follows immediately from Theorem \ref{strong} because forests
are class I.     \end{proof}
 
\section{Acknowledgement}

The second author wishes to thank the University of Newcastle for support from a
2020 - 2021 Summer Research Scholarship during which time this research was begun.

\end{document}